\documentclass[12pt,a4paper]{amsart}
\usepackage{amsfonts}
\usepackage{amsthm}
\usepackage{amsmath}
\usepackage{amscd}
\usepackage[parfill]{parskip}
\usepackage[latin2]{inputenc}
\usepackage{t1enc}
\usepackage[mathscr]{eucal}
\usepackage{indentfirst}
\usepackage{graphicx}
\usepackage{graphics}
\usepackage{enumerate}
\usepackage{pict2e}
\usepackage{epic}
\usepackage{mathtools}

\usepackage{verbatim}
\newtheorem{theorem}{Theorem}[section]
\newtheorem{definition}[theorem]{Definition}
\newtheorem{remark}[theorem]{Remark}

\newtheorem{proposition}[theorem]{Proposition}
\newtheorem{corollary}[theorem]{Corollary}
\newtheorem{algorithm}[theorem]{Algorithm}
\newtheorem{lemma}[theorem]{Lemma}
\newtheorem*{lemma*}{Lemma}

\DeclarePairedDelimiter\floor{\lfloor}{\rfloor}
\DeclarePairedDelimiter\set{\{}{\}}
\numberwithin{equation}{section}
\usepackage[margin=2.9cm]{geometry}
\usepackage{epstopdf}

\begin{document}

\title{On card guessing game with one time riffle shuffle and complete feedback}

\author[Pengda Liu]{Pengda Liu}
\date{March 2020}
\address{Department of Mathematics, Stanford University, Stanford, CA 94305} 

\email{pengda@stanford.edu}

\subjclass[2010]{60C05, 65C50}

\keywords{card guessing, dovetail shuffle, asymptotic analysis}

\begin{abstract} This paper studies the game of guessing riffle-shuffled cards with  complete feedback. A deck of $n$ cards labelled 1 to $n$ is riffle-shuffled  once and placed on a table. A player tries to guess the cards from top and is given complete feedback after each guess. The goal is to find the guessing strategy with maximum reward (expected number of correct guesses). We give the optimal strategy for this game and prove that the maximum expected reward is $n/2+\sqrt{2/\pi}\cdot\sqrt{n}+O(1)$, partially solving an open problem of Bayer and Diaconis \cite{lair}.
\end{abstract}

\maketitle

\section{Introduction}
Consider the following game in which a deck of $n$ cards is shuffled according to some method. A player is then asked to guess the shuffled deck from top down. Each round, the player only guesses the card at the top of the deck and receives certain feedback after the guess. Then the top card is removed and the player continues to guess the next card until the deck is empty. The \textit{reward} of the game is the total number of correct guesses. The goal is  to find the optimal guessing strategy so that the expected reward is maximized. The best strategy would then depend on the \{shuffling method, feedback\} combination. This problem is  studied in \cite{graham} in the framework of sequential experiments and is shown to have great applications in statistical testings like taste testing and partially randomized clinical trials. Three common feedbacks that they consider are
\begin{enumerate}[.]
\item 
\textbf{zero feedback\\}
In this scenario, no feedback of any kind is given.
\item
\textbf{correctness feedback\\}
In this scenario, after each guess, the player is told whether the guess is correct or not.
\item
\textbf{complete feedback\\}
In this scenario, after each guess, the correct card is revealed to the player.
\end{enumerate}
They mostly focus on the uniform shuffling setting. In this paper, we consider the game of \textit{dovetail shuffle} (also called \textit{riffle shuffle} and we will use the two names interchangeably) with complete feedback. In  \cite{lair} Bayer and Diaconis  raised  the open problem of deciding the optimal guessing strategy with complete feedback when a deck has been riffle shuffled an arbitrary number of times and calculating the expected reward. We make the first known progress towards this problem by completely solving the case where the cards are  riffle shuffled once. This already requires new techniques and the asymptotic formula for the expected reward appears to be highly non-trivial. We present our main results next.
\subsection{Main Results}
Consider the following algorithm for playing the game.
\begin{algorithm}
\label{riffle_algo}
\begin{enumerate}[1.]
\item (guessing consecutive stage) Guess 1 at the first step. After that, if the previously revealed card sequence $[1,2,...,m]$ is still consecutive, then we guess the next to be card $m+1$.
\item
(separating pile stage) Suppose that when card $k$ is shown, the revealed sequence $1,2...m,k$ becomes non-consecutive, 
separate the remaining cards into two sequences $A=\set{m+1,...,k-1}$ and $B=\set{k+1,...,n}$.
\item
(guess interleave stage) Each time guess the first card in the longer sequence and cross the revealed card from its sequence.
\end{enumerate}
\end{algorithm}
\begin{remark}
When $m+1>k-1$, we let $A$ to be empty and if $k+1>n$, we let $B$ to be empty.
\end{remark}
We show that the algorithm is optimal and give its expected return as the next theorem.
\begin{theorem}
\label{riffle_main}
Algorithm~\ref{riffle_algo} is the optimal strategy for guessing  a deck of $n$ cards riffle shuffled once with complete feedback and has expected reward $\frac{1}{2}\cdot n+\sqrt{2/\pi}\cdot\sqrt{n}+O(1)$.
\end{theorem}
This will be proved in section 3 and next we introduce some notations to be used in the paper.
\subsection{Notations}
For an event or statement $A$, we use $\delta(A)$ to denote the indicator function. For two functions $f(n)$ and $g(n)$ of non-negative integers, $f\sim g$ means $\lim_{n\rightarrow\infty} f/g=1$, $f=O(g)$ means  $|f|\leq C|g|$ for some universal constant $C$ and $f=o(g)$ means $\lim_{n\rightarrow\infty} f/g=0$. For two probability distributions $P_1$ and $P_2$ on a finite set $\Omega$, the \textit{total variation distance} $TV(P_1, P_2)$ is defined to be
$$TV(P_1,P_2)=\frac{1}{2}\sum_{\omega\in\Omega}|P_1(\omega)-P_2(\omega)|.$$
An equivalent definition is 
$$TV(P_1,P_2)=\frac{1}{2}\sup_{f:\Omega\rightarrow [-1,1]}\set{\sum_{\omega\in\Omega}P_1(\omega)f(\omega)-\sum_{\omega\in\Omega} P_2(\omega)f(\omega)}.$$
We let $[n]$ denote the sequence $1,2,...,n$ and use $S_n$ for the permutation group of $[n]$.
For a sequence $a=a_1,...,a_n$ that are distinct numbers in $[n]$, its \textit{rising sequence} is a maximal increasing consecutive subsequence. Let $rs(a)$ denote the number of rising sequences of $a$. For example, for $a=[1,4,2,5,3,6]$, its rising sequences are $[1,2,3]$ and $[4,5,6]$, thus $rs(a)=2$. For any $\pi\in S_n$, let $rs(\pi)=rs(\pi([n]))$ to be the number of rising sequences of the permutation $\pi$. We will particularly care about the permutations with 2 rising sequences and let $R_n=|\set{\pi\in S_n:rs(\pi)=2}|$.
\subsection{Formulation of shuffling}
Throughout this paper we assume that initially a deck of $n$ cards are laid and labeled as $1,2,...,n$ from top to bottom. A given shuffling method is an algorithm $f$ that upon query outputs a shuffled sequence $(c_1,...,c_n)$ representing that from top to bottom, the shuffled cards are $c_1,...,c_n$. We formulate  such algorithm $f$ as 
follows. First $f$ is a random permutation and thus defines a probability distribution $Q_{f}$ over the symmetry group $S_n$. Each time when we query the algorithm $f$, it samples a permutation $\pi$ from the symmetry group $S_n$ with respect to $Q_f$ and  applies $\pi([n])$ to get the shuffled deck. When there is an need, we use $f_n$ to denote that this algorithm is for $n$ cards. With such formulation,  $f$ also induces an $n$ by $n$ transition matrix $P^f$ where $P^f_{ij}$ is the probability that  card $i$ ends up at position $j$ after the shuffle. When the shuffling method $f$ used in context is unambiguous, we just use $Q$ and $P$.
\subsection{High level proof idea}
If the player guesses with complete feedback, let $\mathcal{U}$ be the set of all sequences of length no more than $n$, then the optimal strategy is a map $\lambda:\mathcal{U}\rightarrow [n]$ such that 
$$\lambda(x_{1}...x_{m-1})=\arg\max_a\Pr(x_m=a|x_{1}...x_{m-1})\hspace{0.5cm}\forall x_{1}....x_{m-1}\in\mathcal{U}.$$
Let $R^*(\pi, f)$ be the reward with this optimal strategy when the sampled permutation is $\pi$, then we have the maximum expected reward satisfies
\begin{equation}
\label{direct}
R^*(f)=\sum_{\pi\in S_n}R^*(\pi, f)Q_f(\pi).
\end{equation}
Directly computing these conditional probabilities and summing over all $S_n$ may be applicable to some cases like uniform shuffle. For other cases, it may be  hard to compute all the conditional probabilities to arrive at an operational description of the optimal strategy and we sometimes we need to adopt the following recursive approach. The best guess for the first card is always $\max_{j}\set{P^f_{j1}}$. After the first card is revealed to be $i$, we can think of the remaining $n-1$ cards as being shuffled according to a conditional shuffle $f^{(i)}$ and 
\begin{equation}
\label{recursive}
R^*(f)=\max_{j}\set{P^f_{j1}}+\sum_{i=1}^{n}P^f_{i1}R^*(f^{(i)}).
\end{equation}\\
This is exactly the main idea behind the proof---the riffle shuffle can be approximated by another shuffling method with recursive structure.
\section{Related Work}
\subsection{The mathematical model for riffle shuffle}
Dovetail shuffle is a common shuffling method used in real life. The shuffler first divides the deck into two piles and then drops the cards from the two piles to make a new shuffled pile. We will use the following GSR model introduced by Gilbert and Shannon (see
Gilbert (1955)) and  Reeds (1981) to mathematically describe the dovetail shuffling algorithm.

\textbf{GSR(Gilbert-Shannon-Reeds) Model}
\begin{enumerate}[(1)]
\item
The  deck of $n$ cards is cut at a position  into two piles randomly according to the binomial distribution such that the probability of cutting it at position $k$ (where $0\leq k\leq n$)  is $\frac{\binom {n}{k}}{2^n}$. After this step, we end up with two piles  of cards $A,B$ of size $k$ and $n-k$.
\item
Start with the shuffled pile $C$ being empty and 
drop the cards from the bottom of  $A$ and $B$
onto $C$ one at a time in the following way. At each time, let $x$ be the number of remaining cards in pile $A$ and $y$ be the number of remaining cards in pile $B$, drop a  card from pile $A$ with probability $\frac{x}{x+y}$ and a card from pile $B$ with probability $\frac{y}{x+y}$.
\end{enumerate}

Diaconis \cite{group} gives an analysis showing that the GSR model is a
good model for riffle shuffle  in real life.  Diaconis \cite{dev} briefly mentioned another way to interpret step (2) without proof, we prove it here for completeness as the following lemma.
\begin{lemma}
\label{interleave_lemma}
Step (2) is equivalent to interleaving the two piles uniform randomly.
\end{lemma}
\begin{proof}
Suppose initially pile $A$ has $a$ cards and pile $B$ has $b$ cards.
For a sequence $x_1...x_m$, let $P_1(.)$ be the probability that, with step (2), this sequence ends up being the first $m$ cards in the shuffled pile from bottom to top and let $P_2$ be defined similarly if we interleave the two piles uniformly. Define $P_1$ and $P_2$ both be 1 when $m=0$. We prove by induction on $m$ that $P_1(x_1....x_m)=P_2(x_1....x_m)$.\\
The base case $m=0$ holds by definition. Now suppose the result holds for all sequences of length $m$, consider any sequence $x_1,...,x_{m+1}$ that could be the layout of the first $m+1$ cards of the shuffled deck, that is, $rs(x_1,...,x_{m+1})\leq 2$. Let 
$$a'=|A\setminus\set{x_1,...,x_m}|,b'=|B\setminus\set{x_1,...,x_m}|.$$
Let $P_1(x_{m+1}|x_1,...,x_m)$ denote the probability that the $m+1^{th}$ card being $x_{m+1}$ conditioning on the first $m$ cards being $x_1,...,x_m$. Note that conditioning on step 1, $x_{m+1}$ can assume only two values, either the first card from pile $A$ or the first card from pile $B$. WLOG, suppose $x_{m+1}$ is from pile $A$, then we have
$$P_1(x_1...x_{m+1})=P_1(x_{m+1}|x_1,...,x_m)P_1(x_1,...,x_m)=\frac{a'}{a'+b'}P_1(x_1,...,x_m).$$
Now consider the case where we interleave the two piles uniformly, given that the first $m$ cards are $x_1,...,x_m$, there are $\binom{a'+b'}{a'}$ ways to interleave the remaining cards and $\binom{a'+b'-1}{a'-1}$ ways to interleave if the next card is $x_{m+1}$, thus
\begin{align*}
&P_2(x_{m+1}|x_1,...,x_m)=\frac{\binom{a'+b'-1}{a'-1}}{\binom{a'+b'}{a'}}=\frac{a'}{a'+b'}\\
\Rightarrow &P_2(x_1,...,x_{m+1})=P_2(x_{m+1}|x_1,...,x_m)P_2(x_1,...,x_m)=\frac{a'}{a'+b'}P_2(x_1,...,x_m).   
\end{align*}
By induction, we have $P_1(x_1,...,x_{m+1})=P_2(x_1,...,x_{m+1})$ thus step (2) is equivalent to interleave the two piles uniformly.
\end{proof}
Given a deck already cut, the probability of each of the configuration is given by the following corollary.
\begin{corollary}
\label{interleave_corollary}
If the deck is cut at position $k$ , the resulting deck's configuration is one of the possible interleavings each with probability $\frac{1}{\binom{n}{k}}$.
\end{corollary}
The GSR model gives an operational description of riffle shuffle, now we provide its distribution $Q$ and transition matrix $P$. First from the operational description we note that every permutation $\pi$ in the support of $Q$ satisfies $rs(\pi)=1$ or 2, that is, it is either the identity or has exactly 2 rising sequences. Bayer and Diaconis \cite{lair}  proves a more general result about $a$-shuffle where the original deck is separated into $a$ packs instead of 2. Here we use the version of $a=2$ and a simpler proof is provided for completeness.
\begin{theorem}
\label{riffle_Q}
For dovetail shuffle, the distribution $Q$ on $S_n$ satisfies
$$
Q(\pi)=
     \begin{cases}
       \frac{n+1}{2^n}& \text{if } \pi = id\\
       \frac{1}{2^n} & \text{if } rs(\pi)=2 \\
       0& \text{otherwise}.\\
     \end{cases}
$$
\end{theorem}
\begin{proof}
Let $\Pr(\text{cut}_k)$ be the probability of cutting at position $k$ and  $\Pr(\pi|\text{cut}_k)$ be the probability of $\pi$ conditioning on that it is cut at $k$. By Lemma~\ref{interleave_lemma} and Corollary~\ref{interleave_corollary} we have
\begin{align*}
Q(\pi) & =\sum_{k=0}^{n}\Pr(\text{cut}_k)\Pr(\pi|\text{cut}_k) \\
&=\sum_{k=0}^n\frac{\binom{n}{k}}{2^n}\frac{1}{\binom{n}{k}}\delta(\text{there is an interleaving to generate }\pi([n])\text{ when cut at }k)\\
&=\frac{M(\pi)}{2^n}
\end{align*}
where $M(\pi)$ is the number of cuts that can generate $\pi([n])$.\\
Every cut could generate the identity, so  $M(id)=n+1$. For the other  $\pi$ such that $rs(\pi)=2$, exactly one cut could generate it (which is to cut it at the minimum of the ends of the two rising sequences), thus $M(\pi)=1$, so the theorem is proved.
\end{proof}
\begin{remark}
This theorem shows that the number of rising sequences is a sufficient statistics and this is true in general for $a$-shuffle as well. We also have that the entropy of $Q$ is $n\log2-\frac{n+1}{2^n}\log(n+1).$
\end{remark}
\begin{corollary}
\label{rs}
Recall that in section 1.2 we let $R_n$ denote the number of permutations with 2 rising sequences. From this theorem we have that $R_n=(1-Q(id))2^n=2^n-n-1$.
\end{corollary}
In \cite{ciucu} Ciucu has made great contributions to guessing riffle shuffled deck with zero feedback in which he proves the following theorem
\begin{theorem}
\label{main_ciucu}
The optimal strategy for guessing  a deck of $n$ cards riffle shuffled once without feedback has expected reward $\frac{2}{\sqrt{\pi}}\cdot\sqrt{n}+O(1)$.
\end{theorem}
In his proof, the transition matrix $P$ for riffle shuffle is calculated and this plays a major role in our proof.
\begin{theorem}
\label{riffle_P}
For dovetail shuffle, the transition matrix $P$ satisfies
$$
P_{ij}=
     \begin{cases}
       \frac{1}{2^j}\binom{j-1}{j-i}& \text{if } i<j\\
       \frac{1}{2^n}(2^{j-1}+2^{n-j}) & \text{if } i=j \\
       \frac{1}{2^{n-j+1}}\binom{n-j}{i-j}& \text{if } i>j.\\
     \end{cases}
$$
\end{theorem}

\section{Proof of Theorem~\ref{riffle_main}}
The proof consists of two parts. We first show that the algorithm is optimal in Proposition~\ref{algorithm_optimal} and prove its expected reward in Proposition ~\ref{algo_return}.
\begin{proposition}\label{algorithm_optimal}
Algorithm~\ref{riffle_algo} is the optimal strategy.
\end{proposition}
In showing the optimality, we first prove a lemma explaining the \textit{guessing consecutive stage}.
\begin{lemma}
\label{consecutive}
Given a deck of $n$ cards riffle shuffled once, let the true cards be $a_1,....,a_n$ from top down, then we have
$$\arg\max_l \Pr(a_m=l|a_i=i\text{ for }1\leq i<m)=m.$$
\end{lemma}
\begin{proof}
Let the true underlying permutation be $\pi$ and $$I_m=\set{\pi|\pi([m])=id([m])\text{ and }rs(\pi)=2}$$ 
then by Corollary~\ref{rs} we have
$$|I_m|=|\set{\pi|rs(\pi[m+1,...,n])=2}|=R_{n-m}=2^{n-m}-n+m-1.$$
By Theorem~\ref{riffle_Q}
$$\Pr(\pi([m])=id([m]))=\Pr(\pi=id)+\sum_{\pi\in I_m}\Pr(\pi)=\frac{n+1}{2^n}+|I_m|\frac{1}{2^n}=\frac{2^{n-m}+m}{2^n}.$$
Thus
\begin{align*} \Pr(\pi(m)=m|\pi([m-1])=id([m-1]))&=\frac{\Pr(\pi([m])=id([m]))}{\Pr(\pi([m-1])=id([m-1]))}\\
&=\frac{2^{n-m}+m}{2^{n-m+1}+m-1}>\frac{1}{2}
\end{align*}
which gives  $\arg\max_l P(a_m=l|a_i=i\text{ for }1\leq i<m)=m$. 
\end{proof}
With the above lemma, it's easy to prove that our algorithm is optimal.

\begin{proof}[Proof of Proposition ~\ref{algorithm_optimal}]
By Theorem~\ref{riffle_P}, we have $P_{11}=\frac{1+2^{n-1}}{2^n}>\frac{1}{2}$, thus at the first step, the best guess should be 1. By Lemma~\ref{consecutive} and induction, if the previous revealed sequence is consecutive, the best guess should be the next card in the consecutive sequence. Finally, if the revealed sequence $[1,2,....,m, k]$ became non-consecutive when $k$ is flipped, by step 2 of GSR model and Lemma~\ref{interleave_lemma}, we know that the remaining deck is formed by interleaving $A=\set{m+1,...,k-1}$ and $B=\set{k+1,...,n}$ uniformly. Then we have
$$\Pr(\text{next card=}m+1)=\frac{\binom{|A|+|B|-1}{|A|-1}}{\binom{|A|+|B|}{|A|}}=\frac{|A|}{|A|+|B|}.$$
Thus in this stage the best guess is  the beginning of the longer sequence.
\end{proof}
The next proposition gives the expected reward.
\begin{proposition}\label{algo_return}
Algorithm~\ref{riffle_algo} has expected reward $n/2+\sqrt{2/\pi}\cdot\sqrt{n}+O(1).$
\end{proposition}
In proving this proposition, we first introduce the following  sequence which was also appears in  Ciucu \cite{ciucu}.
\begin{definition}
\label{sequence_def}
Let the sequence $\set{a_i}_{i=0}^{\infty}$ be such that $a_i=\frac{1}{2^i}\binom{i}{\floor{i/2}}$.
\end{definition}
\cite{ciucu} proves a lemma about $\sum_{i=0}^{\infty}a_i$ and we prove something similar with a more refined asymptotic bound. 
\begin{lemma}
\label{sequence_sum}
Let $A_n=\sum_{i=0}^na_i$, then $A_n=c\sqrt{n}+O(1)$ where the constant $c=\sqrt{\frac{8}{\pi}}\approx1.60$.
\end{lemma}
\begin{proof}
By Sterling's formula of the form $n!=\sqrt{2\pi n}(\frac{n}{e})^ne^{\lambda_n}$ where $\frac{1}{12n+1}<\lambda_n<\frac{1}{12n}$, we have
\begin{align*}
a_{2i}&=\frac{(2i)!}{2^{2i}i!i!}=\frac{\sqrt{4\pi i}(\frac{2i}{e})^{2i}e^{\lambda_{2i}}}{2^{2i}\Big(\sqrt{2\pi i}(\frac{i}{e})^ie^{\lambda_i}\Big)^2}\\
&=\frac{e^{\lambda_{2i}}}{\sqrt{\pi i}e^{2\lambda_i}}=\frac{1}{\sqrt{\pi i}}e^{\lambda_{2i}-2\lambda_i}.   
\end{align*}
By the bound on $\lambda_n$, we have
\begin{align*}
&\frac{1}{24i+1}-\frac{2}{12i}<\lambda_{2i}-2\lambda_i<\frac{1}{24i}-\frac{2}{12i+1}\\
&\Rightarrow \lambda_{2i}-2\lambda_i=\Theta(\frac{1}{i})\Rightarrow e^{\lambda_{2i}-2\lambda_i}=1+O(\frac{1}{i})\\
&\Rightarrow a_{2i}=\frac{1}{\sqrt{\pi i}}+O(\frac{1}{i^{\frac{3}{2}}}).  
\end{align*}
where the second last equality is by the expansion of $e^x=\sum_{i=0}^{\infty}\frac{x^i}{i!}$.\\
We have $a_{2i}=a_{2i-1}$  and since the  series $\sum_{i=0}^n\frac{1}{i^{\frac{3}{2}}}$ converges,
$$A_{2n}=\sum_{i=1}^{2n}a_{i}+O(1)=2\sum_{i=1}^na_{2i}+O(1)=\frac{2}{\sqrt{\pi}}\sum_{i=1}^{n}\frac{1}{\sqrt{i}}+O(1).$$
To estimate the series $\sum_{i=1}^{\infty}\frac{1}{\sqrt{i}}$, we have
\begin{align*}
 &2\sqrt{n+1}-2=\int_{1}^{n+1}\frac{1}{\sqrt{x}}dx\leq \sum_{i=1}^n\frac{1}{\sqrt{i}}\leq 1+\int_{1}^n\frac{1}{\sqrt{x}}dx=2\sqrt{n}-1\\
&\Rightarrow A_{2n}=\sqrt{\frac{8}{\pi}}\sqrt{2n}+O(1).
\end{align*}
Since $A_{2n}\leq A_{2n+1}\leq A_{2n+2}$, we have $A_n=\sqrt{\frac{8}{\pi}}\cdot\sqrt{n}+O(1)$.
\end{proof}

Let $f_n$ denote the riffle shuffle algorithm for $n$ cards, now we attempt to compute the maximal expected reward $R^*(f_n)$. It's hard to directly compute this using Equation~\eqref{direct}. Instead, we use the recursive Equation~\eqref{recursive}. For $k\geq 2$, the conditional shuffling $f^{(k)}_n$ is just interleaving uniformly which is an easy object to deal with (recall that in section 1.4 we define $f^{(k)}_n$ to the random permutation on the remaining $n-1$ cards given that the first card is $k$). If $k=1$, then we note that the strategy on the remaining cards is a relabelling of Algorithm~\ref{riffle_algo} on $n-1$ cards. Define $g_n=f^{(1)}_{n+1}$ to be the shuffling(a random permutation) on the remaining $n$ cards when the first card is revealed to be 1 after the riffle shuffle.
Then  $g_{n-1}=f^{(1)}_n$ shares the same optimal strategy with $f_{n-1}$ after reducing every card number by 1. Thus we conjecture that $g_n$ and $f_n$ may be very similar and indeed this conjecture is true in the following sense.
\begin{lemma}
\label{fg}
\begin{enumerate}[(a)]
\item 
The optimal guessing algorithm for $g_n$ is the same as $f_n$.
\item
$R^*(g_n)=R^*(f_n)+o(1)$.
\end{enumerate}
\end{lemma}
\begin{proof}
For (a), since Algorithm~\ref{riffle_algo} is optimal by Proposition~\ref{algorithm_optimal}, the optimal strategy for $g_n$ is that algorithm applied when the first card is 1. Thus it is isomorphic to the optimal strategy for $f_{n}$ after reducing every card number by 1.\\
For (b), given any permutation $\pi$ on $[n]$, we denote $\pi'$ to be a permutation on $[n+1]$ such that $\pi'(1)=1$ and $\pi'(i)=\pi(i-1)+1$. Basically $\pi'$ acts on $[2,...,n+1]$ the same as $\pi$ acts on $[1,...,n]$.\\
First, we compute the distribution $Q_{g_n}$ on $S_n$. By Theorem~\ref{riffle_Q} and Theorem~\ref{riffle_P}, we have
\begin{align*}
 Q_{g_n}(id_{[n]})&=\text{Pr}_{f_{n+1}}(id_{[n+1]}|\text{ the first card is 1})\\
 &=\frac{Q_{f_{n+1}}(id_{[n+1]})}{P^{f_{n+1}}_{11}}=\frac{n+2}{2^n+1}. 
\end{align*}
For any $\pi$ with two rising sequences, $\pi'$ also satisfies $rs(\pi')=2$, thus we have
$$Q_{g_n}(\pi)=\text{Pr}(\pi'|\pi'(1)=1)=\frac{Q_{f_{n+1}}(\pi')}{P^{f_{n+1}}_{11}}=\frac{1}{2^n+1}.$$
Thus the total variation distance of their distributions satisfies
\begin{align*}
TV(Q_{f_n},Q_{g_n})&=\frac{1}{2}\sum_{\pi\in S_n}|Q_{f_n}(\pi)-Q_{g_n}(\pi)|=\frac{1}{2}|Q_{f_n}(id)-Q_{g_n}(id)|+\frac{1}{2}\sum_{\pi\in I_n}|Q_{f_n}(\pi)-Q_{g_n}(\pi)|\\
&=\frac{1}{2}|Q_{f_n}(id)-Q_{g_n}(id)|+\frac{1}{2}|I_n|\cdot|\frac{1}{2^n}-\frac{1}{2^n+1}|=\frac{2^n-n-1}{(2^n+1)2^n}.    
\end{align*}
Recall that in section 1.4 we defined $R^*(\pi, f_n)$ as the reward of the optimal strategy for $f_n$ when the underlying permutation is $\pi$. Part $(a)$ tells us that $R^*(\cdot, f_n)=R^*(\cdot, g_n)$ as functions on $S_n$. Since $R^*(\cdot, f_n)\leq n$, by the equivalent definition of total variation distance we have
\begin{align*}
|R^*(f_n)-R^*(g_n)|&=|\sum_{\pi\in S_n}R^*(\pi,f_n)Q_{f_n}(\pi)-\sum_{\pi\in S_n}R^*(\pi,f_n)Q_{g_n}(\pi)|\\
&\leq 2n\cdot TV(Q_{f_n},Q_{g_n})=\frac{2n(2^n-n-1)}{2^n(2^n+1)}\leq\frac{n}{2^{n-1}}=o(1).\\
\end{align*}
Thus we have $ R^*(g_n)=R^*(f_n)+o(1)$.
\end{proof}
With the above lemma and Equation~\eqref{recursive} , we can establish a recursive relationship for the sequence $\set{R^*(f_n)}$ and prove our proposition.
\begin{proof}[Proof of Proposition~\ref{algo_return}]
For two piles $A$ and $B$ with size $a$ and $b$ interleaved uniformly, let $f(a,b)$ denote the maximum expected reward. Then by  Equation~\eqref{recursive},
\begin{equation}\label{recur_1}
  R^*(f_n)=P_{11}+P_{11}R^*(g_{n-1})+\sum_{k=2}^nP_{k1}f(k-1,n-k).  
\end{equation}
Let $G(n)=R^*(f_n)$, by Theorem~\ref{riffle_P} and Lemma~\ref{fg} we have
$$G(n)=\frac{1+2^{n-1}}{2^n}\big(1+G(n-1)+o(1)\big)+\sum_{k=2}^n\frac{\binom{n-1}{k-1}}{2^n}f(k-1,n-k).$$
Define
$$S(n)=\sum_{k=1}^n\binom{n}{k}f(k,n-k)\text{  and   }F(n)=\frac{S(n)}{2^{n+1}}.$$
Since $G(n)\leq n$ and $\lim_{n\rightarrow\infty} G_n/2^n=0$, we have
\begin{equation}
\label{recursive_G}
G(n)=\frac{1}{2}G(n-1)+F(n-1)+\frac{1}{2}+o(1) 
\end{equation}
with condition $G(1)=1$.\\
For $f(a,b)$, by considering where the next card comes from, we have 
\begin{equation}
\label{recursive_f}
 f(a,b)=\max\set{\frac{a}{a+b},\frac{b}{a+b}}+\frac{b}{a+b}f(a,b-1)+\frac{a}{a+b}f(a-1,b) \end{equation}
with conditions $f(a,b)=f(b,a)$ and $f(a,0)=f(0,a)=a$.\\
Next, we attempt to solve the above recursive formulas, by Equation~\eqref{recursive_f} we have
\begin{align*}
S(n)&=n+\sum_{k=1}^{n-1}\binom{n}{k}(\max\set{\frac{k}{n},1-\frac{k}{n}}+\frac{k}{n}f(k-1,n-k)+\frac{n-k}{n}f(k,n-k-1))\\
&=n+\sum_{k=1}^{n-1}\binom{n}{k}\max\set{\frac{k}{n},1-\frac{k}{n}}+\sum_{k=1}^{n-1}\binom{n}{k}\frac{k}{n}f(k-1,n-k)+\sum_{k=1}^{n-1}\binom{n}{k}\frac{n-k}{n}f(k,n-k-1)\\
&=n+\sum_{k=1}^{n-1}\binom{n}{k}\max\set{\frac{k}{n},1-\frac{k}{n}}+\sum_{k=1}^{n-1}\binom{n-1}{k-1}f(k-1,n-1-(k-1))\\
&+\sum_{k=1}^{n-1}\binom{n-1}{k}f(k,n-1-k)\\
&=n+\sum_{k=1}^{n-1}\binom{n}{k}\max\set{\frac{k}{n},1-\frac{k}{n}}+2S(n-1)\\
&=n+\sum_{k=1}^{\floor{\frac{n}{2}}}\binom{n}{k}\frac{n-k}{n}+\sum_{k=\floor{\frac{n}{2}}+1}^{n-1}\binom{n}{k}\frac{k}{n}+2S(n-1)\\
&=n+\sum_{k=1}^{\floor{\frac{n}{2}}}\binom{n-1}{k}+\sum_{k=\floor{\frac{n}{2}}+1}^{n-1}\binom{n-1}{k-1}+2S(n-1)\\
&=n+\sum_{k=1}^{n-2}\binom{n-1}{k}+\binom{n-1}{\floor{\frac{n}{2}}}+2S(n-1)\\
&=2^{n-1}+\binom{n-1}{\floor{\frac{n-1}{2}}}+n-2+2S(n-1).
\end{align*}
Thus we end up with the following recursive formula for $S(n)$
$$S(n)=2^{n-1}+\binom{n-1}{\floor{\frac{n-1}{2}}}+n-2+2S(n-1)$$
which gives
\begin{align*}
F(n)&=\frac{1}{4}+\frac{\binom{n-1}{\floor{\frac{n-1}{2}}}+n-2}{2^{n+1}}+F(n-1)\\
&=\frac{1}{4}+\frac{n-2}{2^{n+1}}+\frac{a_{n-1}}{4}+F(n-1)\\
&\Rightarrow F(n)=\frac{n}{4}+\frac{\sum_{k=1}^{n-1}a_k}{4}+O(1).    
\end{align*}
By Lemma~\ref{sequence_sum} we have
\begin{equation}\label{F_recursive}
F(n)=\frac{n}{4}+\sqrt{\frac{n}{2\pi}}+O(1).
\end{equation}
Recall that from earlier we have
\begin{equation}\label{G_recursive}
G(n)=\frac{1}{2}G(n-1)+F(n-1)+\frac{1}{2}+o(1).
\end{equation}
Next, we prove that $G(n)=\frac{1}{2}\cdot n+\sqrt{\frac{2}{\pi}}\cdot\sqrt{n}+O(1)$.\\
By Equation~\eqref{F_recursive}, there is a constant $c_0$ such that $F(n)=\frac{n}{4}+\sqrt{\frac{n}{2\pi}}+x_n$ and $|x_n|\leq c_0$ for all $n$. By Equation~\eqref{G_recursive}, there is a constant $c_1$ such that $G(n)=\frac{1}{2}G(n-1)+F(n-1)+y_n$ and $|y_n|\leq c_1$ for all $n$.\\ 
Suppose $G(n)=\frac{1}{2}\cdot n+\sqrt{\frac{2}{\pi}}\cdot\sqrt{n}+z_n$ let $C=\max\set{2c_0+2c_1+2, |z_1|}$, we prove by induction that $|z_n|\leq C$.\\ By our choice of $C$, the base case of $n=1$ is true. Now suppose the statement holds for $n$, consider $n+1$. Then we have
\begin{align*}
 G(n+1)&=\frac{1}{2}G(n)+F(n)+y_{n+1}=\frac{1}{2}(\frac{1}{2}\cdot n+\sqrt{\frac{2}{\pi}}\cdot\sqrt{n}+z_n)+\frac{n}{4}+\sqrt{\frac{n}{2\pi}}+x_n+y_{n+1}\\
&=\frac{n}{2}+\sqrt{\frac{2}{\pi}}\cdot\sqrt{n}+\frac{z_n}{2}+x_n+y_{n+1}\\
&=\frac{n+1}{2}+\sqrt{\frac{2}{\pi}}\cdot\sqrt{n+1}-\frac{1}{2}-\sqrt{\frac{2}{\pi}}(\sqrt{n+1}-\sqrt{n})+\frac{z_n}{2}+x_n+y_{n+1}\\
&\Rightarrow z_{n+1}=\frac{z_n}{2}+x_n+y_{n+1}-\frac{1}{2}-\sqrt{\frac{2}{\pi}}(\sqrt{n+1}-\sqrt{n})\\
&\Rightarrow |z_{n+1}|\leq|\frac{z_n}{2}|+|x_n|+|y_{n+1}|+\frac{1}{2}+\sqrt{\frac{2}{\pi}}\frac{1}{\sqrt{n}+\sqrt{n+1}}\\
&\leq\frac{C}{2}+c_0+c_1+1\leq C.
\end{align*}
Thus the statement holds for $n+1$, which means $|z_n|\leq C$ for all $n$, thus we have the expected reward is
$$G(n)=n/2+\sqrt{2/\pi}\cdot\sqrt{n}+O(1).$$
\end{proof}
\begin{remark}
Currently our approximation technique is not refined enough to find a constant $c$ to replace the $O(1)$ term by $c+o(1)$. By doing numerical calculation for $1\leq n\leq 10000$, the error term is below 0.5.
\end{remark}
\section{Open problems}
The notation in this section is the same as in section 3.
\begin{enumerate}[1.]
\item 
Calculating the variance of the number of correct guesses under optimal strategy. Since we already have a formula for the expected value $R^*(f_n)$, it suffices to compute the following expected square sum
\begin{equation}
T^*(f_n) := \sum_{\pi\in S_n}R^*(\pi,f_n)^2Q_f(\pi).    
\end{equation}
It satisfies
$$
T^*(f_n)=P_{11}(1+R^*(g_{n-1}))^2+\sum_{k=2}^nP_{k1}f(k-1,n-k)^2$$
$$=\frac{1+2^{n-1}}{2^n}(1+R^*(f_{n-1}))^2+\sum_{k=2}^n\frac{\binom{n-1}{k-1}}{2^n}f(k-1,n-k)^2+o(1).$$
Thus it remains to estimate
$$F'(n):=\sum_{k=2}^n\binom{n-1}{k-1}f(k-1,n-k)^2.$$
Currently we are not able to  get a good estimate of $F'$ since an expansion using Equation~\eqref{recursive_f} would involve cross terms and be very complicated to solve.
\item
Generalize to $k$ riffle shuffles $f^{(k)}_n$. A natural next step is to prove an optimal guessing strategy and find its expected number of correct guesses for a deck of cards riffle shuffled for $k\geq2$ times. A conjectured optimal strategy is given in ~\cite{lair}. The main difficulty in extending the current method to the general setting is the following. With $k=1$, the two piles of cards before interleaving would be fully determined once the first non-consecutive card appears. For larger $k$, the natural generalization of this would be based on the sequential description of $a$-shuffle formulated in~\cite{lair} on page 299. However, the original card piles before interleaving are much harder to be determined by  the revealed cards if $k$ is large.
\item
Generalize to yes or no feedback. For the case of uniform shuffle, it is solved in~\cite{graham}. However, this question remains wide open for riffle shuffle, primarily because the conditional shuffle of receiving negative feedback is less understood.
\end{enumerate}
\section{Acknowledgements}
We would like to  thank Pr.Diaconis for teaching a class at Stanford University through which I get to know this interesting topic and providing some references. We thank Anita for proofreading and providing valuable feedback. We also thank the anonymous reviewer for helpful comments and suggestions.

\end{document}